\documentclass[preprint,12p,times]{elsarticle}
\usepackage{amsmath}
\usepackage{amsthm}
\usepackage{amsfonts}
\usepackage[utf8]{inputenc}
\usepackage{amssymb}
\newtheorem{problem}{Problem}[section]
\usepackage{graphicx}
\usepackage{subfigure}
\usepackage{multirow}
\usepackage{enumitem}
\usepackage{color}%
\usepackage{natbib}
\usepackage{epstopdf}

\newtheorem{mydef}{Definition}[section]
\newtheorem{mytheo}{Theorem}[section]
\newtheorem{myrem}{Remark}[section]
\newtheorem{mylem}{Lemma}[section]

\begin{document}

\begin{frontmatter}
\title{A novel semi-analytical multiple invariants-preserving integrator for conservative PDEs}
\author[Shi]{Wei Shi}\ead{shuier628@163.com}
\author[Wang]{Bin Wang}\ead{wangbinmaths@xjtu.edu.cn}
\author[Liu]{Kai~Liu\corref{cor1}}\ead{laukai520@163.com}

\address[Shi]{School of Physical and Mathematical Sciences, Nanjing Tech University, Nanjing 211816, P.R.China}
\address[Wang]{School of Mathematics and Statistics, Xi'an Jiaotong University, 710049 Xi'an, P.R. China}
\address[Liu]{Department of Mathematics, Nanjing Audit University, 211815 Nanjing, P.R. China}
 \cortext[cor1]{Corresponding author}

\begin{abstract}
Many conservative partial differential  equations such as  the  Korteweg-de Vries (KdV) equation, and the nonlinear Schr\"{o}dinger equations, the Klein-Gordon equation have more than one invariant functionals. In this paper, we propose the definition of the discrete variational derivative, based on which,  a novel semi-analytical multiple invariants-preserving integrator for the conservative partial differential  equations is constructed  by projection technique. The proposed integrators are shown to have the same order of accuracy as the underlying integrators. For applications, some concrete mass-momentum-energy-preserving integrators are derived for the KdV equation.
\end{abstract}

\end{frontmatter}
\vskip0.5cm \pagestyle{myheadings}
\thispagestyle{plain} \markboth{}{}

\section{Introduction}
It is known that   partial differential equation (PDE) plays an important role in science and engineering. It can describe many phenomena in physics, engineering, chemistry, other sciences.  Much attention has been paid to investigate the  analytical solutions of  partial differential equations \cite{demiray2005, halim2004, Gegenhasi2007, JLzhang2003}. The investigation on analytical solutions  can offer the  physicists and engineers a powerful tool to examine the feasibility of the model by adjusting some physical parameters, and give good enough support to numerical simulation. However, generally speaking, the analytical solutions of PDEs are not obtainable. Hence, the development of numerical integrators for PDEs is demanded.

Many PDEs such as the KdV equation, and the nonlinear Schr\"{o}dinger equations, the Klein-Gordon equation can be expressed in nonlinear Hamiltonian form. An important feature of Hamiltonian systems is that they admit conservation law structures which  are fundamental to the derivation of analytical solutions, the analysis of the qualitative behaviours, and the numerical discretization of the systems. It has become  common practice that
numerical integrators should be designed to retain the  conservation law structures or other geometric structures, which will be more preferable when studying the long-time behaviour of dynamical systems. Such numerical integrators are usually called geometric or
structure-preserving. We refer the reader to
\cite{hairer2002,McLachlan2002,Sanz-Serna1992} for recent surveys of
this research.  In this paper, we  focus ourselves on the energy/invariants-preserving integrator, which is a typical branch of structure-preserving integrators.
For  ordinary differential equations (ODEs),  a variety of  invariant-preserving integrators, such as the continuous-stage Runge-Kutta(-Nystr\"{o}m) (RK(N))
integrators \cite{Hairer2010EnergyPreservingVO}, the discrete gradient integrators \cite{Quispel1996-1,cie2011,dahlby2011,celledoni2012}, the Hamiltonian boundary value methods \cite{Brugnano2013,Brugnano2011,Brugnano2009}, have been developed in
 relatively general frameworks. In comparison, the construction of the  invariant-preserving integrators
 for PDEs seems more complicated since PDEs are a huge and motley collection of
problems and a case-by-case discussion is required to devise  the invariant-preserving integrators for  each partial differential equation under consideration
  (see, e.g. \cite{celledoni2012,chen2001,cai2013}). Some progress has been made to give
a fairly general framework to develop invariant-preserving integrators for PDEs.
In \cite{celledoni2012}, by using the method of lines and   the average vector
field method, a systematic procedure to construct invariant-preserving schemes for evolutionary PDEs is developed. By reformulating  the PDEs into multi-symplectic Hamiltonian forms, many enery-preserving or multi-symplectic methods are
derived (see, e.g.,
\cite{chen2011,gong2014,bridges2006,hongjl2006}). Furihata, Matsuo et al. presented the concept of discrete variational derivatives, based on which, finite-difference schemes that inherit energy conservation property are derived for PDEs \cite{furihata1999,matsuo2001,furihata2001}. In \cite{dahlby2011-2}, a general procedure for constructing linearly implicit conservative numerical integrators for PDEs is presented.  All the  procedures mentioned above require  the semidiscretization of the PDEs in space.

In this paper, we  consider invariant-preserving integrators for PDEs in a semi-analytical framework. To be more precise, we focus on  time-stepping numerical integrators and  do not require the PDEs to be discretised in spatial direction. The benefit is that it does not depend on the number of  independent variables and the order of derivatives in space. Therefore, it is applicable to  a wider range of PDE models. We firstly extend the concept of discrete gradient for  gradient of a function to the variational derivative of a functional. Then, a semi-analytical discrete variational derivative integrator will be derived for the conservative PDEs. The variational derivative integrator preserves the energy exactly. Furthermore, multiple invariants-preserving integrators for conservative PDEs that have more than one conservation laws will also be constructed by using projection.
The outline of this paper is as follows. Some preliminaries  are
presented and the definition of the discrete variational derivative is proposed in Section \ref{sec:pre}.  In Section \ref{sec: DVDI}, the novel semi-analytical energy-preserving integrators and the multiple invariants-preserving integrators are constructed based on the discrete variational derivative and projection.  Some properties of the proposed integrators are discussed as well. For applications, some concrete invariants-preserving integrators are constructed for the KdV equation in Section \ref{sec:KdV}. The last section
focuses on some conclusions and discussions.

\section{Preliminaries and the discrete variational derivative } \label{sec:pre}
We consider nonlinear first-order conservative PDE with the Hamiltonian formulation
\begin{equation}
\left\{\begin{aligned}
&\dfrac{\partial u}{\partial t}=\mathcal{J}\dfrac{\delta \mathcal
{G}}{\delta u},\\
&u(t_0)=u_0,
\end{aligned}\right.
\label{waveq}%
\end{equation}
where $\mathcal{J}$ is a skew adjoint operator, the energy functional

\begin{equation}\label{rightg}
\mathcal {G}[u]=\int_{\Omega}G[u]dx, \ \ \ \Omega\subseteq
\mathbb{R}^d,
\end{equation}
and $u:\mathbb{R}^d\times[t_0,+\infty]\rightarrow \mathbb{R}$,
$dx=dx_1\cdots dx_d$.  We use the square brackets in \eqref{rightg}
to indicate that the energy  functional $\mathcal{G}$ and the  local energy  density $G$  depend on the function $u$ as well as
derivatives of  $u$ with respect to the independent variables
$x=(x_1,\cdots,x_d)$ up to  some degree $\nu$. The
variational derivative $\dfrac{\delta \mathcal {G}}{\delta u}$  is  defined  by the relation
 {
 \begin{equation}\label{variationder}
\int_{\Omega}\dfrac{\delta \mathcal {G}}{\delta u}
vdx=\dfrac{d}{d\epsilon}\rvert
_{\epsilon=0}\mathcal{G}[u+\epsilon v],
\end{equation}}
for any sufficiently smooth  function $v(x)$.

For the case of $d=1$, suppose
\begin{equation*}
\mathcal {G}[u]=\int_{\Omega}G\left(u,\frac{\partial u}{\partial
x},\ldots,\frac{\partial^\nu u}{\partial x^\nu}\right)dx.
\end{equation*}
Then it can be derived that
\begin{equation*}
\dfrac{\delta \mathcal {G}}{\delta u}=\dfrac{\partial G}{\partial
u}-\dfrac{\partial}{\partial x}\left(\dfrac{\partial G}{\partial
u_x}\right)+\dfrac{\partial}{\partial x^2}\left(\dfrac{\partial
G}{\partial u_{xx}}\right)+\cdots+(-1)^\nu\dfrac{\partial}{\partial
x^\nu}\left(\dfrac{\partial G}{\partial u^{(\nu)}}\right).
\end{equation*}
For the case of  $d\geqslant 2$, one can apply the the Euler operator to $G[u]$  to obtain the variational derivatives (see e.g. \cite{dahlby2011-2}
for details).

From now on,   the solution of \eqref{waveq}  is assumed to have sufficient regularity
and  the equipped boundary conditions on $\Omega$
 of \eqref{waveq} satisfies that the boundary terms vanish when calculating
 integration by parts (for example, periodic boundary conditions, zero Dirichlet boundary conditions). Furthermore, we denote $\mathcal{F}(\Omega)\subset L^2(\Omega)$ as the function space that the solution $u(\cdot,t)$ lies in.
By our assumption, the equation of the form \eqref{waveq} has in
common the energy conservation property
\begin{equation}
\dfrac{d}{d t}\mathcal {G}[u]=0.
\end{equation}
The key idea to construct  energy-preserving integrators for \eqref{waveq} is to introduce the concept of discrete variational derivative (DVD).

\begin{mydef}
The function $\frac{\delta \mathcal{G}}{\delta(u, v)}\in\mathcal{F}(\Omega)$ is a discrete variational derivative  of the functional $\mathcal{G}$ provided that  for  any   functions $u, v\in\mathcal{F}(\Omega)$, $u\neq v$, satisfying

\begin{equation}\label{dvd}
\left\{
\begin{aligned}
&\mathcal{G}[u]-\mathcal{G}[v] =\int_{\Omega} \frac{\delta \mathcal{G}}{\delta(u, v)}\cdot(u-v) \mathrm{d} x:=\left<\frac{\delta \mathcal{G}}{\delta(u, v)},u-v\right>, \\
 &\frac{\delta \mathcal{G}}{\delta(u, u)} =\frac{\delta \mathcal{G}}{\delta u},
\end{aligned}\right.
\end{equation}
where $<\cdot,\cdot>$ denotes the inner product of $L^2(\Omega)$:
$$<v,w>=\int_{\Omega}vwdx.$$
\end{mydef}
The discrete variational derivative can be regarded as a continuous generalization of discrete gradient for the gradient of a function.
The concept of discrete gradient leads to the discrete gradient methods for ordinary differential equations (ODEs).  We refer the reader to
\cite{gonz1996,Itoh1988,Quispel1996-1,Quispel1996-2,McLachlan1999,quispel2008,cie2011} for more research on this topic.

The simplest discrete variational derivative of  $\mathcal{G}$ is
\begin{equation}\label{dvd1}
\frac{\delta \mathcal{G}}{\delta(u, v)}=\dfrac{G[u]-G[v]}{u-v}.
\end{equation}

Similar argument as the average vector field (AVF), one of the frequently used discrete gradients,  yields the  AVF-type discrete variational derivative:

\begin{equation}\label{avfdvd}
\frac{\delta \mathcal{G}_{\mathrm{AVF}}}{\delta(u, v)}=\int_0^1 \frac{\delta \mathcal{G}}{\delta u}[\xi u+(1-\xi) v] \mathrm{d} \xi.
\end{equation}
As a matter of fact, we can verify that

$$
\begin{aligned}
\mathcal{G}[u]-\mathcal{G}[v]&=\int_0^1 \frac{\mathrm{d}}{\mathrm{d} \xi} \mathcal{G}[\xi u+(1-\xi) v] \mathrm{d} \xi\\
 & =\int_0^1\frac{\mathrm{d}}{\mathrm{d} \varepsilon}\rvert_{\varepsilon=0} \mathcal{G}[v+(\xi+\varepsilon)(u-v)]\mathrm{d}\xi \\
& =\int_0^1\int_{\Omega} \frac{\delta \mathcal{G}}{\delta u}[\xi u+(1-\xi) v]\cdot(u-v) \mathrm{d}x\mathrm{d}\xi\\
&=\int_{\Omega} \int_0^1\frac{\delta \mathcal{G}}{\delta u}[\xi u+(1-\xi) v]\mathrm{d}\xi\cdot(u-v) \mathrm{d}x.
\end{aligned}
$$
Therefore, \eqref{avfdvd} is indeed a discrete variational derivative of $\mathcal{G}[u]$.

\begin{myrem}
The concept of discrete variational derivative is different from the ''discrete variational derivative''  given in \cite{furihata1999}. It has appeared in \cite{dahlby2011-2}  but has not been  discussed under the analytical framework in details.
\end{myrem}
\section{The semi-analytical discrete variational derivative integrator for \eqref{waveq} }\label{sec: DVDI}
Based on the discrete variational derivatives, we can construct the semi-analytical discrete variational derivative integrator for \eqref{waveq}. The semi-analytical DVD integrator takes the form:
\begin{equation}\label{DG}
\dfrac{u^{k+1}(x)-u^{k}(x)}{\Delta t}=\mathcal{J}\frac{\delta \mathcal{G}}{\delta(u^{k+1}(x), u^{k}(x))}, k=0, 1,\ldots,
\end{equation}
where $u^k(x)$ is an approximation to the exact solution $u(x,t_k)$ at $t_k=t_0+k\Delta t$ which is obtained by  $k$ steps of a time-stepping numerical integrator.
Using discrete variational derivatives \eqref{dvd1} and \eqref{avfdvd} yields two concrete semi-analytical DVD integrators:
\begin{equation}\label{DG1}
\dfrac{u^{k+1}(x)-u^{k}(x)}{\Delta t}=\mathcal{J}\dfrac{G[u^{k+1}(x)]-G[u^{k}(x)]}{u^{k+1}(x)-u^{k}(x)}, k=0, 1,\ldots,
\end{equation}

\begin{equation}\label{DG2}
\dfrac{u^{k+1}(x)-u^{k}(x)}{\Delta t}=\mathcal{J}\int_0^1 \frac{\delta \mathcal{G}}{\delta u}[\xi u^{k+1}(x)+(1-\xi) u^{k}(x)] \mathrm{d} \xi, k=0, 1,\ldots,
\end{equation}
\begin{myrem}
Here, for simplicity, we assume that the skew skew adjoint operator $\mathcal{J}$ is independent of the function $u$. Otherwise, the discretization of the skew adjoint operator $\mathcal{J}[u]$ can be taken as $\mathcal{J}[\dfrac{u^{k+1}(x)+u^{k}(x)}{2}]$ in the  semi-analytical discrete variational derivative integrator \eqref{DG}.
\end{myrem}
Typically speaking, the conservative PDE  \eqref{waveq} may have more than one conserved functionals. Correspondingly, the PDE \eqref{waveq} has more than one Hamiltonian formulations. We may apply the semi-analytical DVD integrator to the particular Hamiltonian formulation corresponding to the functional we want to preserve. Unfortunately,
the semi-analytical DVD integrator  \eqref{DG} cannot preserve more than one conserved functionals at the same time in general. We will address this issue in what follows. Assume that the equation \eqref{waveq} possesses $n$ independent invariant functionals:
\begin{equation}\label{ninvariants}
\mathcal {H}_1[u]=\int_{\Omega}H_1[u]dx, \quad \mathcal {H}_2[u]=\int_{\Omega}H_2[u]dx,  \cdots, \mathcal {H}_n[u]=\int_{\Omega}H_n[u]dx.
\end{equation}
Our target is to construct a semi-analytical integrator that can preserve all the invariants  \eqref{ninvariants}. Due to the conservation of the invariants \eqref{ninvariants},
the solution of system \eqref{waveq} lies on the submanifold
$$M=\{u\in \mathcal{F}(\Omega)  : \mathcal {H}_1(u)=\mathcal {H}_1(u_0), \mathcal {H}_2(u)=\mathcal {H}_2(u_0), \cdots, \mathcal {H}_n(u)=\mathcal {H}_n(u_0)\}.$$
The tangent space $T_uM$ (\cite{dahlby2011}) of $M$ at $u$ is the orthogonal complement space to the linear space	
$$span\left\{\dfrac{\delta \mathcal{H}_1}{\delta u},\dfrac{\delta \mathcal{H}_2}{\delta u},\cdots, \dfrac{\delta \mathcal{H}_n}{\delta u}\right\}\subset L^2(\Omega).$$
where the orthogonality is in the sense of the inner product of $L^2(\Omega)$.

\begin{mydef}
	Let $\frac{\delta \mathcal{H}}{\delta(u, v)}$ be a fixed  discrete variational derivative of $\mathcal{H}[u]$. The discrete tangent space at $(v,w)\in\mathcal{F}(\Omega)\times\mathcal{F}(\Omega)$ is
	\begin{equation*}
	T_{(v,w)}M=\left\{\eta\in\mathcal{F}(\Omega):\left<\frac{\delta \mathcal{H}_1}{\delta(v, w)}, \eta\right>=\left<\frac{\delta \mathcal{H}_2}{\delta(v, w)},\eta\right>=\cdots=\left<\frac{\delta \mathcal{H}_n}{\delta(v, w)},\eta\right>=0\right\}.
	\end{equation*}
	A vector $\eta=\eta_{(v,w)}\in T_{(v,w)}M$ is called a discrete tangent vector.
\end{mydef}

The following lemma plays an important role in deriving the multiple invariants-preserving integrators. The statement and the proof are similar to that of Lemma 2.2 in \cite{dahlby2011}.
\begin{mylem}\label{lem:cons}
	Let $u^{k+1}(x)=\varphi_h(u^k(x))$ be a semi-analytical time-stepping integrator for the equation \eqref{waveq}. It preserves the $n$ invariants \eqref{ninvariants} simultaneously in the sense that
\begin{equation*}
\begin{aligned}&\mathcal{H}_1(u^{k+1}(x)) = \mathcal{H}_1(u^k(x)),\quad \mathcal{H}_2(u^{k+1}(x)) = \mathcal{H}_2(u^k(x)),\\
 &\quad\cdots, \mathcal{H}_n(u^{k+1}(x)) = \mathcal{H}_n(u^k(x))\quad k=0, 1, 2, \ldots,\end{aligned}\end{equation*}
providing that
	\begin{equation*}
	\eta_{(u^{k+1}(x),u^k(x))}:=\dfrac{u^{k+1}(x)-u^k(x)}{\Delta t}\in T_{(u^{k+1}(x),u^k(x))}M.
	\end{equation*}
\end{mylem}
\begin{proof}
	Since $\eta_{(u^{k+1},u^k)}\in T_{(u^{k+1},u^k)}M$,  it can be verified that
	\begin{equation*}
\begin{aligned}&\mathcal{H}_i(u^{k+1}(x))-\mathcal{H}_i(u^k(x))\\
&=\int_{\Omega}\frac{\delta \mathcal{H}_i}{\delta(u^{k+1}(x), u^{k}(x))}\cdot(u^{k+1}(x)-u^k(x))dx\\
&=\Delta t\left<\frac{\delta \mathcal{H}_{i}}{\delta(u^{k+1}(x), u^{k}(x))},\eta_{(u^{k+1}(x),u^k(x))}\right>=0, \\
 &k=0, 1, 2, \cdots, i=1, 2, \cdots, n.
\end{aligned}
	\end{equation*}
\end{proof}
In what follows, we give a general framework to construct the multiple invariants-preserving integrator by using the projection technique.
Let
{\small
$$Y\left(u^{k+1}(x), u^{k}(x)\right)=span\left\{\dfrac{\delta \mathcal{H}_1}{\delta(u^{k+1}(x), u^{k}(x))}, \dfrac{\delta \mathcal{H}_2}{\delta(u^{k+1}(x), u^{k}(x))},\cdots, \dfrac{\delta \mathcal{H}_n}{\delta(u^{k+1}(x), u^{k}(x))}\right\}$$}
be the subspace spanned by the discrete variational derivatives of $\mathcal{H}_{i}[u], i=1, 2, \cdots, n$ at $\left(u^{k+1}(x), u^{k}(x)\right)$. Assume that $\{w^{1}(x), w^2(x),\cdots, w^{n}(x)\}$ is an orthogonal basis of  $Y\left(u^{k+1}(x), u^{k}(x)\right)$ which can be obtained by  the classical Gram-Schmidt procedure.
 Then the projection operator
 $$\mathcal{P}\left(u^{k+1}(x), u^{k}(x)\right)v(x)=v(x)-\sum\limits_{i=1}^n\left<v(x),w^{i}(x)\right>w^{i}(x)$$
 would be a smooth orthogonal projection operator onto the discrete tangent space $T_{\left(u^{k+1}(x), u^{k}(x)\right)} M$. We propose the projection integrator

\begin{equation}\label{projmethod}
y^{k+1}(x)=\psi_h\left(u^k(x)\right), \quad u^{k+1}(x)=u^k(x)+\mathcal{P}\left(u^{k+1}(x), u^{k}(x)\right)\left(y^{k+1}(x)-u^k(x)\right),
\end{equation}
or equivalently
\begin{equation}\label{projmethod2}
  u^{k+1}(x)=u^k(x)+\mathcal{P}\left(u^{k+1}(x), u^{k}(x)\right)\left(\psi_h\left(u^k(x)\right)-u^k(x)\right),
\end{equation}
where $\psi_h$ is the  flow that defines an arbitrary integrator of order $p$. It is easy to see that the projection integrator \eqref{projmethod} satisfies the condition in Lemma 3.1. Hence it  preserves the $n$ invariants \eqref{ninvariants}.

\begin{myrem}
The operator $\mathcal{I}-\mathcal{P}\left(u^{k+1}(x), u^{k}(x)\right)$  is nothing but the orthogonal projection operator on the space $Y\left(u^{k+1}(x), u^{k}(x)\right)$.
\end{myrem}

\begin{myrem}  If both the  discrete variational derivatives $\dfrac{\delta \mathcal{H}_i}{\delta(u^{k+1}, u^{k})}, i=1, 2, \cdots, n$ and the underlying integrator $\psi_h$ are symmetric, then the integrator \eqref{projmethod} is symmetric as well.
\end{myrem}
Using Runge-Kutta (RK) integrator as the underlying integrator $\psi_h$, we can construct concrete multiple invariants-preserving integrators.
\begin{mydef}An s-stage RK integrator for the equation \eqref{waveq} reads

\begin{equation}\label{RKmetheod}
\left\{\begin{array}{l}
U^{k,i}(x)=u^k(x)+\Delta t \sum\limits_{j=1}^s a_{i j} f\left( U^{k,j}(x)\right), \quad i=1, \ldots, s, \\
u^{k+1}(x)=u^k(x)+\Delta t \sum\limits_{i=1}^s b_i f\left( U^{k,i}(x)\right), \quad k=0, 1, \ldots,
\end{array}\right.
\end{equation}
where $a_{i j}, b_i, c_i, i, j=1, \ldots, s$ are real constants, $f(u)=\mathcal{J}\dfrac{\delta \mathcal
{G}}{\delta u}$, $u^k(x)$ denotes the numerical solution after $k$ steps of the integrator and is an approximation to the exact solution $u(x,t_k)$, while the internal stage value $U^{k,i}(x)$ is an approximation to $u\left(x,t_k+c_i \Delta t\right)$.
\end{mydef}
The RK integrator \eqref{RKmetheod} can be briefly expressed by the following Butcher tableau\\
\begin{center}\begin{tabular}{c|l}
$c$ & $A$ \\
\hline & $b^{\top}$
\end{tabular}$=$\begin{tabular}{c|ccc}
$c_1$ & $a_{11}$ & $\ldots$ & $a_{1 s}$ \\
$\vdots$ & $\vdots$ & $\ddots$ & $\vdots$ \\
$c_s$ & $a_{s 1}$ & $\ldots$ & $a_{s s}$ \\
\hline & $b_1$ & $\ldots$ & $b_s$
\end{tabular}
\end{center}
where $b=\left(b_1, \ldots, b_s\right)^{\top}$ and $c=\left(c_1, \ldots, c_s\right)^{\top}$ are $s$-dimensional vectors, and $A=\left(a_{i j}\right)$ is an $s \times s$ matrix.  If $a_{i j}=0$ for all $1 \leq i \leq j \leq s$, the integrator \eqref{RKmetheod} is explicit, otherwise it is implicit.
\begin{mydef}
The projection RK integrator for the equation \eqref{waveq} reads
\begin{equation}\label{ProjectionRKmetheod}
\left\{\begin{array}{l}
U^{k,i}(x)=u^k(x)+h \sum\limits_{j=1}^s a_{i j} f\left( U^{k,j}(x)\right), \quad i=1, \ldots, s, \\
u^{k+1}(x)=u^k(x)+h \mathcal{P}\left(u^{k+1}(x), u^{k}(x)\right)\sum\limits_{i=1}^s b_i f\left( U^{k,i}(x)\right), \quad k=0, 1, \ldots.
\end{array}\right.
\end{equation}
\end{mydef}
It should be noted that no matter the underlying RK integrator is explicit or implicit, the corresponding projection RK integrator will be implicit since $u^{k+1}(x)$ appears in the projection operator.
\begin{mytheo}
If the underlying integrator is of order $p$, then the projection integrator \eqref{projmethod} is of order $p$ as well, i.e.,
$$
\begin{aligned}
& \left\|u(x,t+h)-u(x,t)-\mathcal{P}(u(x,t+h), u(x,t))\left(\psi_h(u(x,t))-u(x,t)\right)\right\|=\mathcal{O}\left(h^{p+1}\right).
\end{aligned}
$$
\end{mytheo}
\begin{proof}
Let $\{w^{1}(x), w^2(x),\cdots, w^{n}(x)\}$ be an orthogonal basis of
\begin{equation*}
\begin{aligned}
&Y\left(u(x,t+h), u(x,t)\right)\\
&=span\left\{\dfrac{\delta \mathcal{H}_1}{\delta(u(x,t+h), u(x,t))}, \dfrac{\delta \mathcal{H}_2}{\delta(u(x,t+h), u(x,t))},\cdots, \dfrac{\delta \mathcal{H}_n}{\delta(u(x,t+h), u(x,t))}\right\},
\end{aligned}
\end{equation*}
then
$$\mathcal{P}\left(u(x,t+h), u(x,t)\right)v(x)=v(x)-\sum\limits_{i=1}^n\left<v(x),w^{i}(x)\right>w^{i}(x).$$

We compute
\begin{equation}\label{proofeq1}
\begin{aligned}
&\left\|u(x,t+h)-u(x,t)-\mathcal{P}(u(x,t+h), u(x,t))\left(\psi_h(u(x,t))-u(x,t)\right)\right\|\\
&=\left\|u(x,t+h)-u(x,t)-\left(\psi_h(u(x,t))-u(x,t)\right)-\sum\limits_{i=1}^n\left<\psi_h(u(x,t))-u(x,t),w^{i}\right>w^{i}\right\|\\
&\leqslant\left\|u(x,t+h)-u(x,t)-\left(\psi_h(u(x,t))-u(x,t)\right)\right\|+\left\|\sum\limits_{i=1}^n\left<\psi_h(u(x,t))-u(x,t),w^{i}\right>w^{i}\right\|\\
&=\left\|u(x,t+h)-u(x,t)-\left(\psi_h(u(x,t))-u(x,t)\right)\right\|+\sum\limits_{i=1}^n\lvert\left<\psi_h(u(x,t))-u(x,t),w^{i}\right>\rvert\\
\end{aligned}
\end{equation}

Since $\psi_h$ is of order $p$, we have

\begin{equation}\label{proofeq2}
\left\|u(x,t+h)-u(x,t)-\left(\psi_h(u(x,t))-u(x,t)\right)\right\|=\left\|u(x,t+h)-\psi_h(u(x,t))\right\|=\mathcal{O}\left(h^{p+1}\right).
\end{equation}

In the following, we give the estimate of

$$
\left<\psi_h(u(x,t))-u(x,t),w^{i}\right>, i=1, 2, \cdots, n.
$$
 Bearing in mind that both
 $$\{w^{1}(x), w^2(x),\cdots, w^{n}(x)\}$$
 and
 $$\left\{\dfrac{\delta \mathcal{H}_1}{\delta(u(x,t+h), u(x,t))}, \dfrac{\delta \mathcal{H}_2}{\delta(u(x,t+h), u(x,t))},\cdots, \dfrac{\delta \mathcal{H}_n}{\delta(u(x,t+h), u(x,t))}\right\}$$
 are the  basses of
$Y\left(u(x,t+h), u(x,t)\right)$, there exist some real constants $d_{ij}$ such that
$$w^{i}(x)=\sum\limits_{j=1}^{n}d_{ij}\dfrac{\delta \mathcal{H}_j}{\delta(u(x,t+h), u(x,t))}, \quad i=1, 2, \ldots, n.$$
Hence
\begin{equation*}\label{proofeq33}
\begin{aligned}
&\left<\psi_h(u(x,t))-u(x,t),w^{i}(x)\right>\\
&=\left<\psi_h(u(x,t))-u(x,t+h)+u(x,t+h)-u(x,t),\sum\limits_{j=1}^{n}d_{ij}\dfrac{\delta \mathcal{H}_j}{\delta(u(x,t+h), u(x,t))}\right>\\
&=\sum\limits_{j=1}^{n}d_{ij}\left<\psi_h(u(x,t))-u(x,t+h),\dfrac{\delta \mathcal{H}_j}{\delta(u(x,t+h), u(x,t))}\right>\\
&\qquad\qquad+\sum\limits_{j=1}^{n}d_{ij}\left<u(x,t+h)-u(x,t),\dfrac{\delta \mathcal{H}_j}{\delta(u(x,t+h), u(x,t))}\right>\\
\end{aligned}
\end{equation*}
According to the definition of discrete variational derivative and the order of $\psi_h$, we have
\begin{equation*}
\begin{aligned}&\left<u(x,t+h)-u(x,t), \dfrac{\delta \mathcal{H}_j}{\delta(u(x,t+h), u(x,t))}\right>\\
&=\mathcal{H}_j[u(x,t+h)]-\mathcal{H}_j[u(x,t)]=0, j=1, 2, \cdots, n\end{aligned}
\end{equation*}
and
$$\|\psi_h(u(x,t))-u(x,t+h)\|=\mathcal{O}(h^{p+1}).$$
Therefore,
\begin{equation}\label{proofeq3}
\left<\psi_h(u(x,t))-u(x,t),w^{i}(x)\right>=\mathcal{O}(h^{p+1}), i=1, 2, \cdots, n.
\end{equation}
The proof is completed by combining the results of \eqref{proofeq1}, \eqref{proofeq2} and \eqref{proofeq3}.
\end{proof}

\section{Application to the KdV equation}\label{sec:KdV}
Various aspects of the  Korteweg-de Vries (KdV) equation have been studied extensively in the literature(\cite{yanjl2019,korteweg1895,thierry2023,gong2021}).  Here, as an example of application to conservative partial differential equations, we consider the KdV equation of the classical form \cite{ascher2004}
\begin{equation}\label{kdv}
\begin{aligned}
	&u_t(x,t)=\alpha u(x,t)u_x(x,t)+\nu u_{xxx}(x,t),\quad (x,t)\in [-l,l]\times[0,T],
\end{aligned}
\end{equation}
with periodic boundary condition
\begin{equation}
u(-l,t)=u(l,t),\quad t\in [0,T]
\end{equation}
and initial condition
\begin{equation}
u(0,x)=u_0(x),\quad x\in[-l,l],
\end{equation}
where $\alpha, \nu$ are real constants.  The KdV equation has a great number of applications in various branches of physical science such as fluid dynamics, aerodynamics, and continuum mechanics \cite{zabusky1967,drazin1989,crighton1995,debnath1998}.
\subsection{Temporal semi-analytical discretzization }
The KdV equation  \eqref{kdv} can be presented in the Hamiltonian form
\begin{equation}\label{kdvHam}\left\{\begin{aligned}
&u_t=\mathcal{J}\frac{\delta \mathcal{G}}{\delta u},\qquad (x,t)\in[-l,l]\times [0,T]\\
&u(-l,t)=u(l,t),\quad t\in [0,T]\\
&u(x,0)=\psi(x),\ \quad x\in[-l,l],
\end{aligned}\right.
\end{equation}
where the skew adjoint operator $\mathcal{J}=\partial_x$  and the Hamiltonian
 \begin{equation}\label{hamiltonianpdf}
 	\begin{aligned}
 		 \qquad \mathcal{G}[u]=\int_{-l}^{l}(\frac{\alpha}{6}u^3-\frac{\nu}{2}u_x^2)dx\equiv:\int_{-l}^{l} G(u,u_x)dx.\\
 	\end{aligned} 	
 \end{equation}
The  the variational derivative of  $\mathcal{G}[u]$ can be computed as
$$\frac{\delta \mathcal{G}}{\delta u}=\frac{\partial G}{\partial u}-\partial_x\left(\frac{\partial G}{\partial u_x}\right)=\frac{\alpha}{2}u^2+\nu u_{xx}.$$

The two semi-analytical DVD integrators proposed in the paper for the KdV equation \eqref{hamiltonianpdf} are

\begin{equation}\label{DG11}
\begin{aligned}
&\dfrac{u^{k+1}(x)-u^{k}(x)}{h}=\partial_x\left(\frac{\alpha}{6}((u^{k+1}(x))^2+u^{k+1}(x)u^{k}(x)+(u^{k}(x))^2)\right.\\
&\left.\qquad\qquad\qquad\qquad\qquad\qquad-\frac{\nu}{2}\frac{(u^{k+1}_x(x))^2-(u^{k}_x(x))^2}{u^{k+1}(x)-u^{k}(x)}\right),
k=0, 1,\ldots
\end{aligned}
\end{equation}
%
and

\begin{equation}\label{DG21}
\begin{aligned}
&\dfrac{u^{k+1}(x)-u^{k}(x)}{h}= \partial_x\left(\frac{\alpha}{6}((u^{k+1}(x))^2+u^{k+1}(x)u^{k}(x)+(u^{k}(x))^2)\right)\\
&\qquad\qquad\qquad\qquad\qquad\qquad\qquad\qquad+\frac{\nu}{2}(u^{k+1}_{xxx}(x)+u^{k}_{xxx}(x)),
k=0, 1,\ldots.
\end{aligned}
\end{equation}

The KdV equation \eqref{kdv}, as a completely integrable system,  actually has  infinite number of invariants. Here, we are concerned with the following three invariant functionals:
 \begin{itemize}
    \item Mass:
   \begin{equation}\label{masspdf}
\mathcal{H}_1[u]  =\int_{-l}^l u d x.
\end{equation}

   \item Momentum:
   \begin{equation}\label{momentumpdf}
\mathcal{H}_2[u]  =\frac{1}{2} \int_{-l}^l u^2 d x.
\end{equation}
   \item Energy:
    \begin{equation}\label{hamiltonianpdf}
 	\begin{aligned}
 		 \mathcal{H}_3[u]=\int_{-l}^l(\frac{\alpha}{6}u^3-\frac{\nu}{2}u_x^2)dx.\\
 	\end{aligned} 	
 \end{equation}
 \end{itemize}
The variational derivatives of $\mathcal{H}_{i}[u], i=1, 2, 3$ can be easily derived as
$$\frac{\delta \mathcal{H}_1}{\delta u}=1, \frac{\delta \mathcal{H}_2}{\delta u}=u, \frac{\delta \mathcal{H}_3}{\delta u}=\frac{\alpha}{2}u^2+\nu u_{xx}.$$
Now, we present a concrete projection RK integrator for the KdV equation \eqref{hamiltonianpdf}. Here, the  AVF-type discrete variational derivative \eqref{avfdvd} is chosen as  the  discrete variational derivative. Then the  projection RK integrator reads
\begin{equation}\label{projmethod1}
 u^{k+1}(x)=u^k(x)+h\mathcal{P}\left(u^{k+1}(x), u^{k}(x)\right)\left(\psi_h(u^k(x))-u^k(x)\right),
\end{equation}
where
$$\mathcal{P}\left(u^{k+1}(x), u^{k}(x)\right)=\mathcal{I}-\mathcal{P}_{Y},$$
with $\mathcal{P}_{Y}$ the orthogonal projection operator on the space
\begin{equation*}
\begin{aligned}&Y\left(u^{k+1}(x), u^{k}(x)\right)=span\left\{1, \frac{u^{k+1}(x)+u^{k}(x)}{2},\right.\\
&\left.\frac{\alpha}{6}((u^{k+1}(x))^2+u^{k+1}(x)u^{k}(x)+(u^{k}(x))^2)+\frac{\nu}{2}(u^{k+1}_{xx}(x)+u^{k}_{xx}(x))\right\}.
\end{aligned}
\end{equation*}

\subsection{Full discretization}\label{Sec:spectralexpansion}
To obtain   full-discretization  schemes for the KdV equation  \eqref{kdv},   a finite-dimension approximation in spatial direction is required. In this section, we will demonstrate how the full-discretization  scheme arises from the  projection RK integrator  \eqref{projmethod1}.
 In what follows,  we adopt the spectral framework from \cite{Brugnano2109KdV,Barletti2024,BrugL2015} to carry out the discretization in spatial direction.

First of all, we expand the solution $u(x,t)$ in space $L^2[-l,l]$ as
\begin{equation}\label{spectralexpansion1}
u(x, t)=\sum_{j \geq 0} u_j(t) \omega_j(x),
\end{equation}
where
$$
\begin{aligned}
\omega_{2 j}(x) & =\sqrt{\frac{2-\delta_{j 0}}{2l}} \cos \left(j \frac{x+l}{l}  \pi\right), \\
\omega_{2 j+1}(x) & =\sqrt{\frac{1}{l}} \sin \left((j+1) \frac{x+l}{l}  \pi\right), \quad j=0,1,2, \ldots
\end{aligned}
$$
is the orthonormal basis \cite{Barletti2024} satisfying
$$
\int_{-l}^l \omega_i(x) \omega_j(x)\mathrm{d} x=\delta_{i j}, \quad i, j=0,1,2, \ldots,
$$
with $\delta_{i j}$  the Kronecker delta. Let
$$
\boldsymbol{\omega}(x)=\left(\begin{array}{c}
\omega_0(x) \\
\omega_1(x) \\
\omega_2(x) \\
\vdots
\end{array}\right), \quad \boldsymbol{u}(t)=\left(\begin{array}{c}
q_0(t) \\
q_1(t) \\
q_2(t) \\
\vdots
\end{array}\right),
$$
the expansion \eqref{spectralexpansion1} can be written compactly as
\begin{equation}\label{expansion}
u(x, t)= \boldsymbol{\omega}(x)^{\top}  \boldsymbol{u}(t).
\end{equation}
Correspondingly, let  the expansions of  $u^{k}(x)$ and  $u^{k+1}(x)$ in  \eqref{projmethod1} be
 \begin{equation}
 u^{k}(x)=\boldsymbol{\omega}(x)^{\top}  \boldsymbol{u}^k,\quad   u^{k+1}(x)=\boldsymbol{\omega}(x)^{\top}  \boldsymbol{u}^{k+1}.
 \end{equation}
It can be verified that
$$
\boldsymbol{\omega}^{(k)}(x)=D^k \boldsymbol{\omega}(x),
$$
where
$$
\begin{aligned}
& D=\frac{ \pi}{l}\left(\begin{array}{cccc}
0 & & & \\
& 1 \cdot J & & \\
& & 2 \cdot J & \\
& & & \ddots
\end{array}\right).
\end{aligned}
$$
Here the skew-symmetric matrix $J$ is defined as
$J=\left(
      \begin{array}{cc}
        0 & 1 \\
        -1 & 0 \\
      \end{array}
    \right)=-J^{\top}.
$
Therefore, the  projection RK integrator  \eqref{projmethod1} now reads
\begin{equation}\label{projmethod11}
\boldsymbol{\omega}(x)^{\top}  \boldsymbol{u}^{k+1}=\boldsymbol{\omega}(x)^{\top}  \boldsymbol{u}^{k}+h\mathcal{P}\left(u^{k+1}(x), u^{k}(x)\right)\left(\psi_h(\boldsymbol{\omega}(x)^{\top}  \boldsymbol{u}^{k})-\boldsymbol{\omega}(x)^{\top}  \boldsymbol{u}^{k}\right),
\end{equation}
where
$$\mathcal{P}\left(u^{k+1}(x), u^{k}(x)\right)=\mathcal{I}-\mathcal{P}_{Y},$$
with $\mathcal{P}_{Y}$ the orthogonal projection operator on the space
\begin{equation*}
\begin{aligned}
&Y\left(u^{k+1}(x), u^{k}(x)\right)\\
&=span\left\{1, \boldsymbol{\omega}(x)^{\top} \frac{\boldsymbol{u}^{k+1}+\boldsymbol{u}^{k}}{2},
\frac{\alpha}{6}((\boldsymbol{\omega}(x)^{\top}  \boldsymbol{u}^{k+1})^2+\boldsymbol{\omega}(x)^{\top}  \boldsymbol{u}^{k+1}\boldsymbol{\omega}(x)^{\top}  \boldsymbol{u}^{k}+(\boldsymbol{\omega}(x)^{\top}  \boldsymbol{u}^{k})^2)+\frac{\nu}{2}\boldsymbol{\omega}(x)^{\top} (D^2)^{\top} (\boldsymbol{u}^{k+1}+\boldsymbol{u}^{k})\right\}\\
&=\boldsymbol{\omega}(x)^{\top}span\left\{\boldsymbol{e},  \frac{\boldsymbol{u}^{k+1}+\boldsymbol{u}^{k}}{2},
\frac{\alpha}{6}\int_{-l}^{l}\boldsymbol{\omega}(x)((\boldsymbol{\omega}(x)^{\top}  \boldsymbol{u}^{k+1})^2+\boldsymbol{\omega}(x)^{\top}  \boldsymbol{u}^{k+1}\boldsymbol{\omega}(x)^{\top}  \boldsymbol{u}^{k}+(\boldsymbol{\omega}(x)^{\top}  \boldsymbol{u}^{k})^2)dx+\frac{\nu}{2} (D^2)^{\top} (\boldsymbol{u}^{k+1}+\boldsymbol{u}^{k})\right\}
\end{aligned}
\end{equation*}
with $\boldsymbol{e}=(1,0,\ldots,0,\ldots)^{\top} $. Here, we  have transferred the projection process in the space $L^2[-l,l]$ to the the infinite dimensional vector space $\mathbb{R}^{\infty}$.

Bearing in mind the orthonormality of the basis functions $\omega_{j}(x), j=0,1,\ldots$, it follows from \eqref{projmethod11} that the  projection RK integrator can be expressed coordinate-wisely as
\begin{equation}\label{projmethod12}
 \boldsymbol{u}^{k+1}= \boldsymbol{u}^{k}+h(I-P)\left(\alpha \int_{-l}^{l}\boldsymbol{\omega}(x)\psi_h(\boldsymbol{\omega}(x)^{\top}  \boldsymbol{u}^{k})dx-  \boldsymbol{u}^{k}\right),
\end{equation}
where $I$ is the infinite dimensional identity matrix and $P$ is the projection matrix on the subspace
\begin{equation}\label{subspace1}span\left\{\boldsymbol{e}_1,  \frac{\boldsymbol{u}^{k+1}+\boldsymbol{u}^{k}}{2},
\frac{\alpha}{6}\int_{-l}^{l}\boldsymbol{\omega}(x)((\boldsymbol{\omega}(x)^{\top}  \boldsymbol{u}^{k+1})^2+\boldsymbol{\omega}(x)^{\top}  \boldsymbol{u}^{k+1}\boldsymbol{\omega}(x)^{\top}  \boldsymbol{u}^{k}+(\boldsymbol{\omega}(x)^{\top}  \boldsymbol{u}^{k})^2)dx+\frac{\nu}{2} (D^2)^{\top} (\boldsymbol{u}^{k+1}+\boldsymbol{u}^{k})\right\},
\end{equation}
which can be calculated as
\begin{equation}
P=G(G^{\top}G)^{-1}G^T,
\end{equation}
where $G$ be a ${\infty\times3}$ matrix  whose columns  consist of the bases of the subspace \eqref{subspace1}.

Then, suitable truncation leads to the concrete practical full-discretization projection RK schemes. To be more precisely, we restrict ourselves to  the  $(2N+1)$-dimension subspace $\mathcal{V}_N=\operatorname{span}\left\{\omega_{j}(x), j=0,1,\ldots,2N.\right\}$ and look for the approximation
\begin{equation}\label{truncatedexpansion}
u(x, t)\thickapprox u_{2N+1}(x,t)=\boldsymbol{\omega}^{\top}_{2N+1}(x)  \boldsymbol{u}_{2N+1}(t)
\end{equation}
with
$$
\boldsymbol{\omega}_{2N+1}(x)=\left(\begin{array}{c}
\omega_0(x) \\
\omega_1(x) \\
\vdots \\
\omega_{2 N}(x)
\end{array}\right) \in\mathbb{R}^{2N+1}, \quad \boldsymbol{u}_{2N+1}(t)=\left(\begin{array}{c}
u_0(t) \\
u_1(t) \\
\vdots \\
u_{2 N}(t)
\end{array}\right) \in \mathbb{R}^{2N+1}.
$$

 The truncated projection RK integrator in space $\mathcal{V}_N$ can be expressed coordinate-wisely as
\begin{equation}\label{truncatedprojmethod12}
 \boldsymbol{u}^{k+1}_{2N+1}= \boldsymbol{u}^{k}_{2N+1}+\Delta t(I_{2N+1}-P_{2N+1})\left(\alpha \int_{-l}^{l}\boldsymbol{\omega}_{2N+1}(x)\psi_h(\boldsymbol{\omega}_{2N+1}(x)^{\top}  \boldsymbol{u}_{2N+1}^{k})dx-  \boldsymbol{u}^{k}_{2N+1}\right),
\end{equation}
where $I_{2N+1}$ is the $2N+1$ dimensional identity matrix,
$$
D_{2N+1}=\frac{ \pi}{l}\left(\begin{array}{cccc}
0 & & & \\
& 1 \cdot J & & \\
& & \ddots & \\
& & & N \cdot J
\end{array}\right)\in\mathbb{R}^{(2N+1)\times (2N+1)}$$
and
\begin{equation}
P_{2N+1}=G_{2N+1}(G_{2N+1}^{\top}G_{2N+1})^{-1}G_{2N+1}^T
\end{equation}
with $G_{2N+1}$ be a ${(2N+1)\times3}$ matrix  whose columns  consist of the bases of the $2N+1$ dimensional truncation of the subspace   \eqref{subspace1}.

Similar to the Projection RK integrator in the space $L^2[-l,l]$, we have the following result for the truncated projection RK integrator \eqref{truncatedprojmethod12} concerning about the preservation of the invariants.
\begin{mytheo}\label{truncatedenergypreservation}
  The truncated projection RK integrator \eqref{truncatedprojmethod12} preserves exactly
the truncated mass
\begin{equation}\label{truncatedmass1}
M_{2N+1}=\int_{-l}^l\boldsymbol{\omega}^{\top}_{2N+1}(x) \boldsymbol{u}_{2N+1}\mathrm{~d} x,
\end{equation}
 the truncated momentum
\begin{equation}\label{truncatedmomentum1}
K_{2N+1}=\int_{-l}^l\left(\boldsymbol{\omega}^{\top}_{2N+1}(x) \boldsymbol{u}_{2N+1}\right)^2 \mathrm{~d} x.
\end{equation}
and the truncated energy
\begin{equation}\label{truncatedenergy1}
\begin{aligned}
H_{{2N+1}}&=\int_{-l}^l \dfrac{1}{6}\alpha \left(\boldsymbol{\omega}^{\top}_{2N+1}(x) \boldsymbol{u}_{2N+1}\right)^3-\frac{1}{2}\nu(\boldsymbol{\omega}^{\top}_{2N+1}(x) D_{2N+1}^{\top}\boldsymbol{u}_{2N+1}  )^2 \mathrm{d} x.
\end{aligned}
\end{equation}
\end{mytheo}

In the following, we numerically solve a KdV equation using the truncated projection RK integrator \eqref{truncatedprojmethod12} to confirm the theoretical results.
\begin{problem}\label{twosoliton}
	Consider the interaction of two solitary waves which are modeled by the KdV equation \eqref{kdv} with the initial condition
$$
\begin{aligned}
& u_0(x)=\frac{12}{\left(1+e^{\theta_1}+e^{\theta_2}+a^2 e^{\theta_1+\theta_2}\right)^2}\left[k_1^2 e^{\theta_1}+k_2^2 e^{\theta_2}+ 2\left(k_2-k_1\right)^2 e^{\theta_1+\theta_2}+a^2\left(k_2^2 e^{\theta_1}+k_1^2 e^{\theta_2}\right) e^{\theta_1+\theta_2}\right],
\end{aligned}
$$
where $a^2=\left(\frac{k_1-k_2}{k_1+k_2}\right)^2=\frac{1}{25}, \theta_1=k_1 x+x_1, \theta_2=k_2 x+x_2$. In this problem, we set $k_1=0.4, k_2=0.6, x_1=4,x_2=15.$ This problem is derived from \cite{yanjl2019}. The parameters in the KdV equation are chosen as $\alpha=-1,\nu=-1$ and the solution region as $x \in[-40 ,40]$.
\end{problem}

\begin{figure}
\centering
	\begin{tabular}
		[c]{cc}%
		\subfigure[Numerical solution]{\includegraphics[width=4cm,height=4cm]{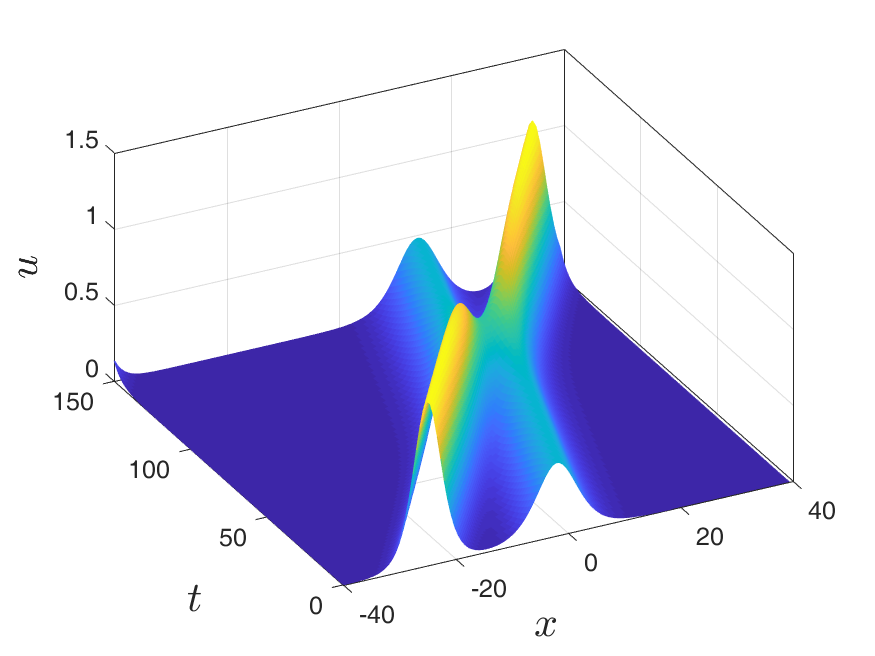}}
		\subfigure[Invariants preservation]{\includegraphics[width=4cm,height=4cm]{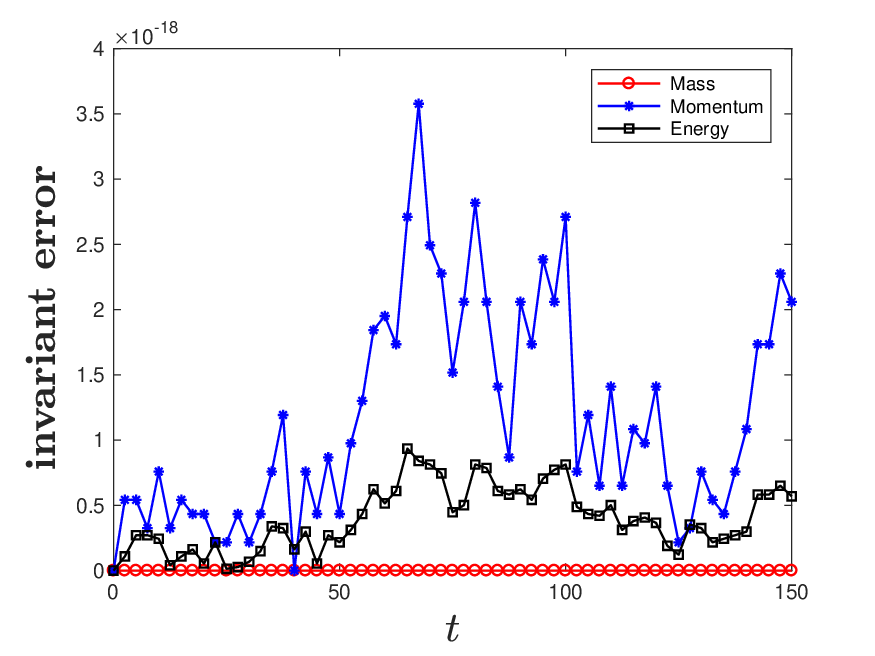}}
        \subfigure[Algebraic order]{\includegraphics[width=4cm,height=4cm]{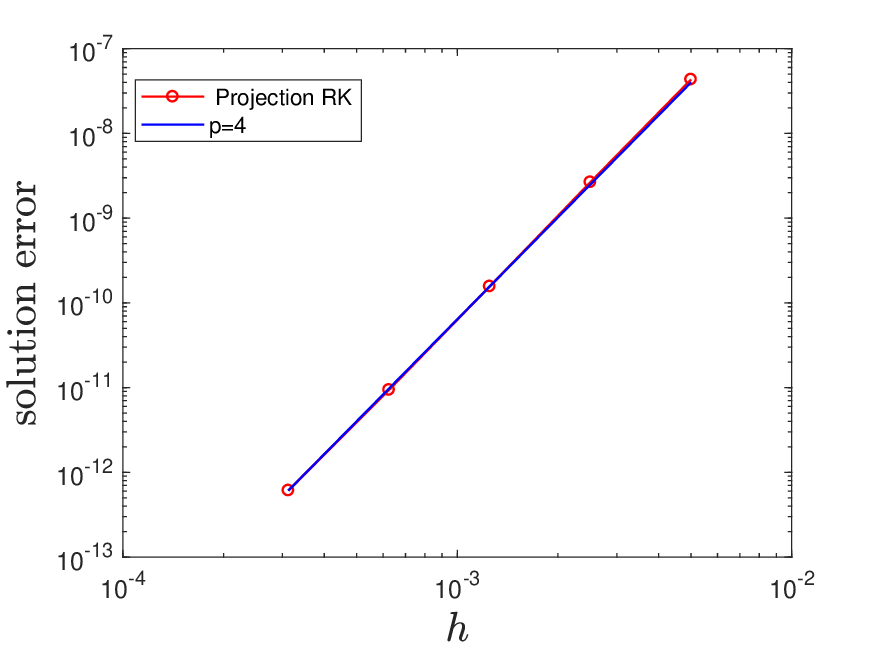}}
	\end{tabular}
	\caption{Numerical results for \textbf {Problem \ref{twosoliton}} with $N=2^6$. (a) the numerical solution obtained by  truncated projection RK  integrator; (b)  preservation of the three invariants; (c) order of convergence.}
	\label{problem2solution}
\end{figure}
We integrate the KdV equation in the interval $t\in[0,150]$ with $N=2^6$ and $h=0.005$.  The  $4$th-order RK integrator with Butcher tableau
$$
\begin{array}{c|cccc}
0 & & & & \\
1 / 2 & 1 / 2 & & & \\
1 / 2 & 0 & 1 / 2 & & \\
1 & 0 & 0 & 1 & \\
\hline & 1 / 6 & 2 / 6 & 2 / 6 & 1 / 6
\end{array}
$$
is chosen as the underlying integrator.
The numerical solution obtained by the truncated projection RK integrator and the preservation of three invariants are shown in Figure \ref{problem2solution}(a)(b). It is observed that the three truncated invariants are preserved by the numerical integrator and the numerical solution is well performed. The two solitary waves move at a constant speed. The velocity of the taller solitary wave is greater than that of the shorter one. At $t=80$, the two solitons overlap, and they are completely apart at $t=120$, having swapped their positions.

Moreover, we integrate the KdV equation in the interval $t\in[0,1]$ using the truncated projection RK integrator with different stepszies. The solution errors versus the stepsizes are displayed in Figure \ref{problem2solution}(c). It can be seen from the numerical result that the projection RK integrator maintains the algebraic order of the underlying RK integrator.

\section{Conclusions and remarks}
In the present paper, by introducing the concept of discrete variational derivative, we obtain a semi-analytical  energy-preserving  discrete variational derivative integrator for Hamiltonian PDEs, which can be viewed as a generalization of  the discrete gradient for Hamiltonian ODEs. Furthermore, semi-analytical multiple invariants-preserving integrators for conservative PDEs are constructed by projection. In this paper,  we focus ourselves on the temporal direction, the obtained integrators are time-stepping. One more step of  a finite-dimension discretization in spatial direction (including suitable approximation to the partial derivatives $\partial_{x}$, $\partial_{xx}$ etc. ) will lead to a full-discretization schemes for the conservative PDEs \eqref{waveq}.
All the analysis in the paper makes perfect sense after  replacing $u^{k}(x)$ by the finite-dimension vector $\mathbf{u}^{k}$ and  the continuous $L^2$ inner product by the  discrete $l^2$ inner product.
This paper offers a new framework for the constructing multiple invariants-preserving integrators for conservative PDEs. The novel approach is conceptually simple, versatile, and helpful for the theoretical analysis of full-discretization energy-preserving schemes.
\section*{Acknowledgements}
The research was supported in part by the Natural Science Foundation of China under Grant 12371403 and 12371433, the National Natural Science Foundation of Jiangsu under Grant BK20200587.


\section*{References}
\begin{thebibliography}{10}

\bibitem{demiray2005}
H.~Demiray.
\newblock A complex travelling wave solution to the {KdV-Burgers} equation.
\newblock {\em Phys. Lett. A}, 344:418--422, 2005.

\bibitem{halim2004}
A.~A. Halim and S.~B. Leble.
\newblock Analytical and numerical solution of a coupled {KdV-MKdV} system.
\newblock {\em Chaos, Solitons Fractals}, 19:99--108, 2004.

\bibitem{Gegenhasi2007}
Gegenhasi and X.~B. Hu.
\newblock {A (2 + 1)-dimensional sine-Gordon equation and its Pfaffian
  generalization}.
\newblock {\em Phys. Lett. A}, 360:439--447, 2007.

\bibitem{JLzhang2003}
J.~L. Zhang and Y.~M. Wang.
\newblock Exact solutions to two nonlinear equations.
\newblock {\em Acta Phys. Sin.}, 52:1574--1578, 2003.

\bibitem{hairer2002}
E.~Hairer, C.~Lubich, and G.~Wanner.
\newblock {\em Geometric Numerical Integration: Structure-Preserving
  Algorithms}.
\newblock Springer-Verlag, Berlin, Heidelberg, 2nd edition, 2006.

\bibitem{McLachlan2002}
R.~I. McLachlan and G.~R.~W. Quispel.
\newblock Splitting methods.
\newblock {\em Acta. Numer.}, 11:341--434, 2002.

\bibitem{Sanz-Serna1992}
J.~M. Sanz-Serna.
\newblock Symplectic integrators for {H}amiltonian problems: an overview.
\newblock {\em Acta Numer.}, 1:243--286, 1992.

\bibitem{Hairer2010EnergyPreservingVO}
E.~Hairer.
\newblock Energy-preserving variant of collocation methods.
\newblock {\em J. Numer. Anal. Ind. Appl. Math.}

\bibitem{Quispel1996-1}
G.~R.~W. Quispel and H.~W. Capel.
\newblock Solving {ODE}s numerically while preserving a first integral.
\newblock {\em Phys. Lett. A}, 218:223--228, 1996.

\bibitem{cie2011}
J.~L. Cie$\acute{s}$li$\acute{n}$ski and B.~Ratkiewicz.
\newblock Energy-preserving numerical schemes of high accuracy for
  one-dimensional {Hamiltonian} systems.
\newblock {\em J. Phys. A: Math. Theor.}, 44:155206, 2011.

\bibitem{dahlby2011}
M.~Dahlby and B.~Owren.
\newblock {A general framework for deriving integral preserving numerical
  methods for PDEs}.
\newblock {\em SIAM J. Sci. Comput.}, 33:2318--2340, 2011.

\bibitem{celledoni2012}
E.~Celledoni, V.~Grimm, R.~I. McLachlan, D.~I. McLaren, D.~O'Neale, B.~Owren,
  and G.~R.~W. Quispel.
\newblock {Preserving energy resp. dissipation in numerical {PDE}s using the
  `Average Vector Field' method}.
\newblock {\em J. Comput. Phys.}, 231:6770--6789, 2012.

\bibitem{Brugnano2013}
F.~Iavernaro and B.~Pace.
\newblock {Line integral methods and their application to the numerical
  solution of conservative problems}.
\newblock {\em math. NA.}, arXiv:1301.2367v1, 2013.

\bibitem{Brugnano2011}
L.~Brugnano, F.~Iavernaro, and D.~Trigiante.
\newblock A note on the efficient implementation of {Hamiltonian} {BVM}s.
\newblock {\em J Comput Appl Math}, 236:375--383, 2011.

\bibitem{Brugnano2009}
L.~Brugnano, F.~Iavernaro, and D.~Trigiante.
\newblock Hamiltonian {BVM}s ({HBVM}s): a family of "drift-free" methods for
  integrating polynomial {Hamiltonian} systems.
\newblock {\em AIP Conf.Proc.}, 1168:715--718, 2009.

\bibitem{chen2001}
Chen.~J. B and Qin.~M. Z.
\newblock {Multi-symplectic Fourier pseudospectral method for the nonlinear
  Schr$\ddot{o}$dinger equation}.
\newblock {\em Electronic Transactions on Numerical Analysis}, 12:193--204,
  2001.

\bibitem{cai2013}
Cai. W, Wang. Y, and Song. Y.
\newblock {Numerical dispersion analysis of a multi-symplectic scheme for the
  three dimensional Maxwell's equations}.
\newblock {\em Journal of Computational Physics}, 234:330--352, 2013.

\bibitem{chen2011}
Y.~Chen, Y.~Sun, and Y.~Tang.
\newblock Energy-preserving numerical methods for {Landau–Lifshitz} equation.
\newblock {\em J. Phys. A: Math. Theor.}, 44:295207, 2011.

\bibitem{gong2014}
Y.~Z. Gong, J.~X. Cai, and Y.~S. Wang.
\newblock Some new structure-preserving algorithms for general multi-symplectic
  formulations of {Hamiltonian PDEs}.
\newblock {\em J. Comput. Phys.}, 279:80--102, 2014.

\bibitem{bridges2006}
T.J. Bridges and S.~Reich.
\newblock Numerical methods for {Hamiltonian} {PDEs}.
\newblock {\em J. Phys. A: Math. Gen.}, 39:5287--5320, 2006.

\bibitem{hongjl2006}
Hong.~J. L, Liu. Y, M.~Hans, and A.~Zanna.
\newblock {Globally conservative properties and error estimation of
  multi-symplectic scheme for Schr$\ddot{o}$dinger equations with variable
  coefficients}.
\newblock {\em Appl. Numer. Math.}, 56:814--843, 2006.

\bibitem{furihata1999}
D.~Furihata.
\newblock Finite difference schemes for $\partial u/\partial
  t=(\partial/\partial x)^a\sigma g/\sigma u$ that inherit energy conservation
  or dissipation property.
\newblock {\em J. Comput. Appl. Math.}, 156:181--205, 1999.

\bibitem{matsuo2001}
T.~Matsuo and D.~Furihata.
\newblock Dissipative or consercative finite-differencr schemes for
  complex-valued nonlinear partial differentiial equations.
\newblock {\em J. Comput. Appl. Math.}, 171:425--447, 2001.

\bibitem{furihata2001}
D.~Furihata.
\newblock Finite-difference schemes for nonlinear wave equation that inherit
  energy conservation property.
\newblock {\em J. Comput. Appl. Math.}, 134:37--57, 2001.

\bibitem{dahlby2011-2}
M.~Dahlby and B.A. Owren.
\newblock A general framework for deriving integral preserving numerical
  methods for {PDEs}.
\newblock {\em SIAM J. sci. comp.}, 33:2318--2340, 2011.

\bibitem{gonz1996}
O.~Gonzalez.
\newblock Time integration and discrete {Hamiltonian} systems.
\newblock {\em J. Nonlinear. Sci.}, 6:449--467, 1996.

\bibitem{Itoh1988}
T.~Itoh and K.~Abe.
\newblock Hamiltonian conserving discrete canonical equations based on
  variational difference quotients.
\newblock {\em J. Comput. Phys.}, 77:85--102, 1988.

\bibitem{Quispel1996-2}
G.~R.~W. Quispel and G.~S. Turner.
\newblock Discrete gradient methods for solving {ODE}s numerically while
  preserving a first integral.
\newblock {\em J. Phys. A: Math. Gen.}, 29:L341--L349, 1996.

\bibitem{McLachlan1999}
R.~I. McLachlan, G.~R.~W. Quispel, and N.~Robidoux.
\newblock Geometric integration using discrete gradients.
\newblock {\em Phil. Trans. R. Soc. London A}, 357:1021--1045, 1999.

\bibitem{quispel2008}
G.~R.~W. Quispel and D.~I. McLaren.
\newblock A new class of energy-preserving numerical integration methods.
\newblock {\em J. Phys. A: Math. Theor.}, 41:045206, 2008.

\bibitem{yanjl2019}
J.L. Yan and L.~H. Zheng.
\newblock { A Class of Momentum-Preserving Fourier Pseudo-Spectral Schemes for
  the Korteweg-de Vries Equation}.
\newblock {\em IAENG Int. J. Appl. Math.}, 49(4):49422, 2019.

\bibitem{korteweg1895}
D.J. Korteweg and G.~de~Vries.
\newblock {On the change of form of long waves advancing in a rectangular
  channel and on a new type of long stationary wave }.
\newblock {\em Philos. Mag.}, 39:422--443, 1895.

\bibitem{thierry2023}
L.~Thierry.
\newblock {Multisolitons are the unique constrained minimizers of the KdV
  conserved quantities }.
\newblock {\em Calc. Var. Partial. Differ. Equ.}, 62( 192, 2023.

\bibitem{gong2021}
Y.Z. Gong, Y.~Chen, C.~W. Wang, and Q.~Hong.
\newblock {A new class of high-order energy-preserving schemes for the
  Korteweg-de Vries equation based on the quadratic auxiliary variable (QAV)
  approach }.
\newblock {\em Numer. Math. Theor. Meth. Appl.}, 15:768--792, 2022.

\bibitem{ascher2004}
U.~M. Ascher and R.~I. McLachlan.
\newblock Multisymplectic box schemes and the {Korteweg de Vries} equation.
\newblock {\em Appl. Numer. Math.}, 48:255--269, 2004.

\bibitem{zabusky1967}
N.J. Zabusky.
\newblock {\em Nonlinear Partial Differential Equations}.
\newblock Academic Press, 1967.

\bibitem{drazin1989}
P.~G. Drazin and R.~S. Johnson.
\newblock {\em Solitons: An Introduction}.
\newblock Cambridge University, New York, 1989.

\bibitem{crighton1995}
D.G. Crighton.
\newblock Applications of {KdV}.
\newblock {\em Acta Appl. Math.}, 39:39--67, 1995.

\bibitem{debnath1998}
L.~Debnath.
\newblock {\em Nonlinear Water Partial Differential Equations for Scientists
  and Engineers}.
\newblock Birkhauser, Berlin, 1998.

\bibitem{Brugnano2109KdV}
L.~Brugnano, G.~Gurioli, and Y.~Sun.
\newblock {Energy-conserving Hamiltonian boundary value methods for the
  numerical solution of the Korteweg-de Vries equation}.
\newblock {\em J. Comput. Appli. Math.}, 351:117--135, 2019.

\bibitem{Barletti2024}
L.~Barletti, L.~Brugnano, G.Gurioli, and F.~Iavernaro.
\newblock { Recent advances in the numerical solution of the nonlinear
  Schr\"{o}dinger Equation}.
\newblock {\em J. Comput. Appli. Math.}, page 115826, 2024.

\bibitem{BrugL2015}
L.~Brugnano, G.F. Caccia, and F.~Iavernaro.
\newblock Energy conservation issues in the numerical solution of the
  semilinear wave equation.
\newblock {\em Appl. Math. Comput.}, 270:842--870, 2015.

\end{thebibliography}
\end{document}